\documentclass[english,leqno]{article}
\usepackage[T1]{fontenc}
\usepackage[latin9]{inputenc}
\usepackage{geometry}
\usepackage{float}
\usepackage{amsmath}
\usepackage{amsthm}
\usepackage{amssymb}
\usepackage{graphicx}
\usepackage{esint}
\usepackage{hhline}
\newcommand{\R}{{\mathbb R}}
\newcommand{\E}{{\mathbb E}}\makeatletter
\newcommand{\N}{{\mathbb N}}
\floatstyle{ruled}
\newfloat{algorithm}{tbp}{loa}
\providecommand{\algorithmname}{Algorithm}
\floatname{algorithm}{\protect\algorithmname}

  \theoremstyle{definition}
  \newtheorem*{condition*}{\protect\conditionname}
\theoremstyle{plain}
\newtheorem{thm}{\protect\theoremname}[section]
  \theoremstyle{plain}
  \newtheorem{prop}[thm]{\protect\propositionname}
  \theoremstyle{definition}
  
  \theoremstyle{plain}
  
  \theoremstyle{plain}
  
  \theoremstyle{plain}
  \newtheorem{remark}[thm]{\protect\remarkname}
  \theoremstyle{plain}
\newtheorem*{assumption*}{\protect\assumptionname}

\usepackage{algorithm,algpseudocode}
\numberwithin{equation}{section}

\makeatother
\usepackage{babel}
\providecommand{\assumptionname}{Assumption}
\providecommand{\conditionname}{Condition}
\providecommand{\corollaryname}{Corollary}
\providecommand{\definitionname}{Definition}
\providecommand{\lemmaname}{Lemma}
\providecommand{\propositionname}{Proposition}
\providecommand{\theoremname}{Theorem}
\providecommand{\remarkname}{Remark}
\begin{document}

\title{Bias behaviour and antithetic sampling in mean-field particle approximations of SDEs nonlinear in the sense of McKean}

\author{O. Bencheikh and B. Jourdain\thanks{Universit\'e Paris-Est, Cermics (ENPC), INRIA, F-77455 Marne-la-Vall\'ee, France. E-mails : benjamin.jourdain@enpc.fr, oumaima.bencheikh@enpc.fr}}
\maketitle
\begin{abstract}
In this paper, we prove that the weak error between a stochastic differential equation with nonlinearity in the sense of McKean given by moments and its approximation by the Euler discretization with time-step $h$ of a system of $N$ interacting particles is ${\mathcal O}(N^{-1}+h)$. We provide numerical experiments confirming this behaviour and showing that it extends to more general mean-field interaction and study the efficiency of the antithetic sampling technique on the same examples.
\end{abstract}

% $\textbf{Keywords}$: Particle systems, non-linear diffusions in the sense of McKean, bias estimation, Monte Carlo methods.

\section{Introduction}

According to \cite{Szn91}, the strong rate of convergence of particle approximations of McKean-Vlasov Stochastic differential equations with Lipschitz coefficients is ${\mathcal O}(N^{-1/2})$ when $N$ denotes the number of particles. This rate is driven by the statistical error and one may wonder whether the bias vanishes quicker. A parallel can be drawn with the time discretization of standard stochastic differential equations where, for Lipschitz coefficients, the strong rate of convergence of the explicit Euler-Maruyama scheme is ${\mathcal O}(\sqrt{h})$ \cite{Ka} with $h>0$ denoting the time-step. Using the Feynman-Kac partial differential equation associated with the stochastic differential equation, Talay and Tubaro \cite{tt} checked that, for smooth coefficients, the weak error behaves in ${\mathcal O}(h)$ and can even be expanded in powers of $h$. In the context of particles interacting through jumps, the ${\mathcal O}(N^{-1})$ behaviour of the bias is known. According to \cite{GM}, for particle approximations of generalized Boltzmann equations, the total variation distance between the law of the path of a particle and the one of the limiting nonlinear Boltzmann process behaves in ${\mathcal O}(N^{-1})$. For Feynman-Kac particle models, expansions in powers of $1/N$ are obtained in \cite{DMPR}. 

The interest in the bias introduced by particle approximations is motivated by numerical efficiency. Indeed, the numerical experiments performed in Section \ref{sec:num} show a general ${\mathcal O}(N^{-1})$ behaviour of the bias, even in models with not so smooth coefficients. Under this behaviour, simulating $\sqrt{N}$ independent copies of the system with $\sqrt{N}$ particles leads to the same order of error (bias with the same order as the ${\mathcal O}(N^{-1/4}\times N^{-1/4})={\mathcal O}(N^{-1/2})$ statistical error) as one simulation of the system with $N$ particles (bias smaller than the ${\mathcal O}(N^{-1/2})$ statistical error). And the former approach is less expensive than the latter as soon as the computational cost of the particle system grows more than linearly with the number of particles. The behaviour of the bias is also of interest in order to adapt to the number of particles the multilevel Monte Carlo method introduced by Giles \cite{Gil08} in the context of time discretization of standard stochastic differential equations. In \cite{HAT}, Haji-Ali and Tempone combine both discretizations through the Multi-index Monte Carlo method.
In this perspective, another interesting question is the possibility to take advantage of the antithetic sampling technique introduced in \cite{Ha} to reduce the variance (see \cite{CGL,BHR} or \cite{Gil15} p57 for the use of this technique in nested Monte Carlo computations). Does the variance of the difference between the empirical mean of the system with $2N$ particles driven by the i.i.d. couples $(X^i_0,W^i)_{1\le i\le 2N}$ ($X^i_0$ and $W^i$ respectively denote the initial condition and the Brownian motion of the $i$-th particle) and the mean of the empirical means of the two independent systems with $N$ particles respectively driven by $(X^i_0,W^i)_{1\le i\le N}$ and $(X^i_0,W^i)_{N+1\le i\le 2N}$ converge quicker to $0$ than ${\mathcal O}(N^{-1})$?

In this paper, using the Feynman-Kac partial differential equation associated with the limiting nonlinear process and its moments, we prove the respective ${\mathcal O}(N^{-1})$ and ${\mathcal O}(N^{-1}+h)$ behaviour of the bias for systems of diffusive particles interacting through moments and their Euler discretization with time step $h$. Of course, the computational cost of such systems is linear in the number $N$ of particles and the above numerical motivation is not valid. Nevertheless this result can be seen as a first step before addressing more general interactions which could necessitate more advanced tools like the master equation for mean-field games introduced by Lions in his seminal lectures at {\it Coll\`ege de France} and studied in \cite{CCD} from a probabilistic point of view. Theorem 3.2 \cite{KT} is proved using the master equation and implies the ${\mathcal O}(N^{-1})$ behaviour of the bias for one-dimensional stochastic differential equations with general interaction in the drift coefficient but no interaction in the diffusion coefficient. We provide numerical experiments showing that the ${\mathcal O}(N^{-1})$ behaviour holds in more general situations including ones with non smooth coefficients. Last, on the same examples, we check that the antithetic variance does not generally decrease quicker than ${\mathcal O}(N^{-1})$.

\section{Estimation of the bias for systems of particles interacting through moments}
Let $\alpha:\R^n\to\R^p$ be Lipchitz continuous and $\sigma=(\sigma^l_j)_{1\le j\le n,1\le l\le d}:[0,T]\times\R^p\times\R^n\to\R^{n\times d}$, $b=(b_j)_{1\le j\le n}:[0,T]\times\R^p\times\R^n\to\R^{n}$ be Lipschitz continuous in their $p+n$ last components uniformly in their first component and such that $\sup_{t\in[0,T]}(|\sigma(t,0,0)|+|b(t,0,0)|)<\infty$.  We consider the stochastic differential equation in dimension $n$ nonlinear in the sense of McKean 
\begin{equation}
   X_t=X_0+\int_0^t\sigma(s,\E[\alpha(X_s)],X_s)dW_s+\int_0^tb(s,\E[\alpha(X_s)],X_s)ds,\;t\in[0,T]\label{edsnl}
\end{equation}
where $X_0$ is some square integrable $\R^n$-valued random variable independent from the $d$-dimensional Brownian motion $(W_t)_{t\ge 0}$. The drift and diffusion coefficient at time $s$ depend on the law of $X_s$ through the moments $\E\left[\alpha\left(X_s\right)\right]$. Since $\sigma $ may only depend on a subset of coordinates of this expectation on $b$ on another subset, moments in drift and diffusion can differ as well as their respective dimensions.
By a solution, we mean a continuous process $(X_t)_{t\in[0,T]}$ adapted to the filtration generated by the Brownian motion $W$ and $X_0$ such that $\sup_{t\in[0,T]}\E[|\alpha(X_t)|]<\infty$ and the above equation holds with the integrals with respect to $dW_s$ and $ds$ making sense. Notice that for any solution, $x\mapsto (\sigma(s,\E[\alpha(X_s)],x),b(s,\E[\alpha(X_s)],x))$ has affine growth uniformly for $s\in[0,T]$. With the square integrability of $X_0$, this implies that $\E\left[\sup_{t\in[0,T]}|X_t|^2
\right]+\sup_{0\le s<t\le T}\frac{\E\left[|X_t-X_s|^2\right]}{t-s}<\infty$ so that, by Lipschitz continuity of $\alpha$, $t\mapsto\E[\alpha(X_t)]$ is H\"older continuous with exponent $1/2$ on $[0,T]$.\\

The  particle approximation of the SDE nonlinear in the sense of McKean \eqref{edsnl} is given by the system with mean-field interaction
\begin{equation}
   X^{i,N}_t=X^i_0+\int_0^t\sigma\bigg(s,\frac 1 N \sum_{j=1}^N\alpha(X^{j,N}_s),X^{i,N}_s\bigg)dW^i_s+\int_0^tb\bigg(s,\frac 1 N \sum_{j=1}^N\alpha(X^{j,N}_s),X^{i,N}_s\bigg)ds,\;1\le i\le N,\;t\in[0,T],\label{systpart}
\end{equation}
with $((X^i_0,W^i))_{i\ge 1}$ i.i.d. copies of $(X_0,W)$. By the Lipschitz assumptions, existence and trajectorial uniqueness hold for this $N\times n$ dimensional equation. The Yamada-Watanabe theorem ensures weak uniqueness and therefore exchangeability of the particles $((X^{i,N}_t)_{t\in[0,T]})_{1\le i\le N}$. Let us also introduce the Euler discretizations with time-step $h>0$ of the SDE \eqref{edsnl} and the particle system : 

\begin{align}
   X^h_t&=X_0+\int_0^t\sigma(\tau^h_s,\E[\alpha(X^h_{\tau^h_s})],X^h_{\tau^h_s})dW_s+\int_0^tb(\tau^h_s,\E[\alpha(X^h_{\tau^h_s})],X^h_{\tau^h_s})ds,\;t\in[0,T]\mbox{ where }\tau^h_s=\lfloor s/h\rfloor h,\label{edsnlh}\\
X^{i,N,h}_t&=X^i_0+\int_0^t\sigma\bigg(\tau^h_s,\frac 1 N \sum_{j=1}^N\alpha(X^{j,N,h}_{\tau^h_s}),X^{i,N,h}_{\tau^h_s}\bigg)dW^i_s+\int_0^tb\bigg(\tau^h_s,\frac 1 N \sum_{j=1}^N\alpha(X^{j,N,h}_{\tau^h_s}),X^{i,N,h}_{\tau^h_s}\bigg)ds,\;1\le i\le N,\;t\in[0,T].\notag
\end{align}
It is natural and convenient to consider that $\tau^0_s=s$, $(X^{0}_t)_{t\in[0,T]}=(X_t)_{t\in[0,T]}$ and $(X^{i,N,0}_t)_{t\in[0,T],1\le i\le N}=(X^{i,N}_t)_{t\in[0,T],1\le i\le N}$ and we use these notations in what follows.

\newpage 

Reasoning like in the laboratory example in \cite{Szn91} or in Theorem 1.3 \cite{JMW}, one easily checks the following result.
\begin{prop}\label{propsznit}
   Strong existence and trajectorial uniqueness hold for the SDE nonlinear in the sense of McKean \eqref{edsnl} and its Euler discretization \eqref{edsnlh}. Moreover $\sup_{h\ge 0}\E\left[\sup_{t\in[0,T]}|X^h_t|^2\right]\le C(1+\E[|X_0|^2])$ where the finite constant $C$ does not depend on $X_0$. Last, if for $i\in\N^*$, $(X^{i,h}_t)_{t\in[0,T]}$ denotes the process obtained by replacement of $(X_0,W)$ by $(X^i_0,W^i)$ in \eqref{edsnlh}, then 
$$\exists C<\infty,\;\forall N\in\N^*,\;\sup_{h\ge 0}\max_{1\le i\le N}\E\left[\sup_{t\in[0,T]}|X^{i,N,h}_t-X^{i,h}_t|^2\right]\le C\frac{1+\E[|X_0|^2]}{N}.$$
\end{prop} 
If for $i\in\N^*$, $(X^{i}_t)_{t\in[0,T]}$ denotes the process obtained by replacement of $(X_0,W)$ by $(X^i_0,W^i)$ in \eqref{edsnl}, this implies the following estimation of the bias introduced by the particle discretization: for any function $\psi:\R^n\to \R$ Lipschitz with constant $L_\psi$ and in particular for each coordinate of $\alpha$, 
\begin{align*}
   \forall s\in[0,T],\;|\E[\psi(X^{1,N}_s)]-\E[\psi(X_s)]|&=|\E[\psi(X^{1,N}_s)]-\E[\psi(X^1_s)]|\le L_\psi\E[|X^{1,N}_s-X^1_s|]\\&\le L_\psi \left(\sup_{h \geq 0} \E\left[\sup_{t\in[0,T]}|X^{1,N,h}_t-X^{1,h}_t|^2\right] \right)^{1/2}\le \sqrt{C}L_\psi\frac{\sqrt{1+\E[|X_0|^2]}}{\sqrt{N}}.
\end{align*}
The first inequality is crude since it prevents cancelations in average and one may wonder whether the bias converges faster to $0$ as $N\to\infty$. Under additional regularity, the answer is positive, which is our main result.
\begin{thm}\label{thmbias}
  Assume that \begin{itemize}\item $\sigma$ is Lipschitz continuous in its $p+n$ last components uniformly in its first component and such that \\ $\sup_{t\in[0,T]}(|\sigma(t,0,0)|+|b(t,0,0)|)<\infty$,\item $\alpha,\psi$ are two times continuously differentiable with bounded derivatives up to the order two and Lipschitz continuous second order derivatives,
\item $a=\sigma\sigma^*$ and $b$ are globally Lipschitz continuous, continuously differentiable with respect to their variables with index in $\{2,\hdots,1+p\}$ with derivatives Lipschitz continuous with respect to their $p+n$ last variables uniformly in their first variable, \item there exists $\tilde d\in\N^*$, $\tilde\sigma:[0,T]\times\R^p\times\R^n\to\R^{n\times \tilde d}$ such that for all $(t,y,x)\in[0,T]\times\R^p\times\R^n$, $a(t,y,x)=\tilde\sigma\tilde\sigma^*(t,y,x)$ and $\tilde\sigma,b$ are continuous and two times continuously differentiable with respect to their $n$ last variables with bounded derivatives up to the order two and second order derivatives Lipschitz continuous with respect to their $n$ last variables uniformly in their $1+p$ first variables,

\end{itemize} Then 
$$\exists C<\infty,\;\forall h\ge 0,\;\forall N\in\N^*,\;\sup_{t\in[0,T]}|\E[\psi(X^{1,N,h}_t)]-\E[\psi(X_t)]|\le C\left(\frac{1}{N}+h\right).$$
\end{thm}

The idea of the proof is, like in \cite{tt}, to use the Feynman-Kac partial differential equation associated with the nonlinear SDE \eqref{edsnl} to first check that the estimation holds for the coordinates of $\alpha$ before concluding that it holds for each test function $\psi$. The following proposition ensures existence and smoothness to this Feynman-Kac PDE which only depends on $\sigma$ through $a=\sigma\sigma^*$. In its statement and its proof based on a stochastic flow approach, we take advantage of the flexibility given by the choice of a square root of $a$ which explains the last assumption in Theorem \ref{thmbias}. Let $(\tilde W_t)_{t\ge 0}$ be a $\tilde d$-dimensional Brownian motion with coordinates $(\tilde W^l_t)_{1\le l\le \tilde d}$ and for $(s,x)\in[0,T]\times\R^n$, $({\tilde X}^{s,x}_t)_{t\in[s,T]}$ denote the solution to 
\begin{equation}
   \label{edsflot}
{\tilde X}^{s,x}_t=x+\int_s^t{\tilde \sigma}(r,\E[\alpha(X_r)],{\tilde X}^{s,x}_r)d{\tilde W}_r+\int_s^tb(r,\E[\alpha(X_r)],{\tilde X}^{s,x}_r)dr,\;t\in[s,T]
\end{equation}
where the coefficients depend on $\E[\alpha(X_r)]$ and not $\E[\alpha({\tilde X}^{s,x}_r)]$.

\begin{prop}\label{propregufk}Let $\psi:\R^n\to\R$ be two times continuously differentiable with bounded derivatives up to the order two and Lipschitz continuous second order derivatives. Under the last assumption in Theorem \ref{thmbias}, for each $t\in[0,T]$, the function $[0,t]\times R^n\ni (s,x)\mapsto u_{t,\psi}(s,x):=\E\left[\psi({\tilde X}^{s,x}_t)\right]$ is once (resp. twice) continuously differentiable with respect to the time (resp. space) variable $s$ (resp. $x$) and such that $u_{t,\psi}$ together with its spatial derivatives up to the order two are Lipschitz continuous in $x$ uniformly in $0\le s\le t\le T$. Moreover, $u_{t,\psi}$ is a classical solution to the Feynman-Kac PDE
\begin{equation}\label{fk}\begin{cases}
   \partial_su_{t,\psi}(s,x)+\frac{1}{2}\sum_{j,k=1}^na_{jk}(s,\E[\alpha(X_s)],x)\partial_{jk}u_{t,\psi}(s,x)+\sum_{j=1}^nb_{j}(s,\E[\alpha(X_s)],x)\partial_{j}u_{t,\psi}(s,x)=0,\;(s,x)\in[0,t]\times\R^n\\
u_{t,\psi}(t,x)=\psi(x),\;x\in\R^n
\end{cases}.
\end{equation}
Here $\partial_j$ and $\partial_{jk}$ denote the partial derivative with respect to the $j$-th coordinate of $x$ and with respect to its $j$-th and $k$-th coordinates.
\end{prop}
\begin{proof}
   Even if this result seems to be well-known, we could not find its proof in the literature. Therefore, we are going to give a sketch of its proof. For notational simplicity, we set ${\tilde \sigma}^0=b$ and ${\tilde W}^0_t=t$.

Let for $j\in\{1,\hdots,n\}$, $e_j$ denote the $j$-th vector of the canonical basis of $\R^n$ and for $l\in\{0,\hdots,\tilde d\}$, $\partial{\tilde \sigma}^l=(\partial_{1+p+k}{\tilde \sigma}^l_j)_{1\le j,k\le n}$  where $\partial_{1+p+k}$ is the partial derivative with respect to the $(1+p+k)$-th variable. Let also for $l\in\{0,\hdots,\tilde d\}$ and $y,z\in\R^n$, $\partial^2{\tilde \sigma}^lyz\in\R^n$ be defined by $(\partial^2{\tilde \sigma}^lyz)_j=\sum_{k,m=1}^n(\partial_{1+p+k}\partial_{1+p+m}{\tilde \sigma}^l_j)y_kz_m$ for $j\in\{1,\hdots n\}$.
 Since $x\mapsto [{\tilde \sigma},b](r,\E[\alpha(X_r)],x)$ has affine growth uniformly in $r\in[0,T]$, standard moment estimations ensure that $$\forall q\ge 1,\;\exists C<\infty,\;\forall (s,x)\in[0,T]\times\R^n,\;\E\left[\sup_{r\in[s,T]}|{\tilde X}^{s,x}_r|^q\right]\le C(1+|x|^q).$$
By \cite{Kun84}
, Theorem 3.3 p223 and its proof, for $r\in[s,T]$, $x\mapsto {\tilde X}^{s,x}_r$ is twice continuously differentiable $\mathbb{P}$-almost surely with derivatives satisfying for $j,k\in\{1,\hdots,n\}$
\begin{align*}
   d\partial_j {\tilde X}^{s,x}_r=&\sum_{l=0}^{\tilde d}\partial{\tilde \sigma}^l(r,\E[\alpha(X_r)],{\tilde X}^{s,x}_r)\partial_j {\tilde X}^{s,x}_rd{\tilde W}^l_r,\; r\in[s,T],\;\partial_j {\tilde X}^{s,x}_s=e_j\\
d\partial_{jk} {\tilde X}^{s,x}_r=&\sum_{l=0}^{\tilde d}\left(\partial{\tilde \sigma}^l(r,\E[\alpha(X_r)],{\tilde X}^{s,x}_r)\partial_{jk} {\tilde X}^{s,x}_r+\partial^2{\tilde \sigma}^l(r,\E[\alpha(X_r)],{\tilde X}^{s,x}_r)\partial_{j} {\tilde X}^{s,x}_r\partial_{k} {\tilde X}^{s,x}_r\right)d{\tilde W}^l_r,\;r\in[s,T],\;\partial_{jk} {\tilde X}^{s,x}_s=0.\end{align*}
Moreover, for any $q\in[1,\infty)$, 
\begin{align*}
  &\sup_{0\le s\le r\le T}\sup_{x\neq y \in\R^n}\sum_{j=1}^n\E\bigg[|\partial_j{\tilde X}^{s,x}_r|^q+\sum_{k=1}^n\bigg(|\partial_{jk}{\tilde X}^{s,x}_r|^q+\frac{|\partial_{jk}{\tilde X}^{s,x}_r-\partial_{jk}{\tilde X}^{s,y}_r|^q}{|x-y|^q}\bigg)\bigg]<\infty\\\mbox{ and }\exists C<\infty,\forall &x\in\R^n,\;\forall 0\le s\le \tilde s\le r\le T,\\&\E\left[|{\tilde X}^{s,x}_r-{\tilde X}^{\tilde s,x}_r|^q+\sum_{j=1}^n\bigg(|\partial_j {\tilde X}^{s,x}_r-\partial_j{\tilde X}^{\tilde s,x}_r|^q+\sum_{k=1}^n|\partial_{jk} {\tilde X}^{s,x}_r-\partial_{jk}{\tilde X}^{\tilde s,x}_r|^q\bigg)\right]\le C(1+|x|^q)(\tilde s-s)^{q/2}.
\end{align*}
These properties ensure that $x\mapsto u_{t,\psi}(s,x)$ is two times continuously differentiable with derivatives $\partial_j u_{t,\psi}(s,x)=\E[\nabla\psi({\tilde X}^{s,x}_t).\partial_j{\tilde X}^{s,x}_t]$ and $\partial_{jk} u_{t,\psi}(s,x)=\E[\nabla\psi({\tilde X}^{s,x}_t).\partial_{jk}{\tilde X}^{s,x}_t+\partial_{k}{\tilde X}^{s,x}_t.\nabla^2\psi({\tilde X}^{s,x}_t)\partial_{j}{\tilde X}^{s,x}_t]$ bounded and Lipschitz continuous in $x$ uniformly in $0\le s\le t\le T$. Moreover, $u_{t,\psi}$ and its spatial derivatives up to the order two are continuous in the time variable.
Let us now check that $u_{t,\psi}$ satisfies the Feynman-Kac PDE \eqref{fk}. Let $s\in(0,t]$ and $\varepsilon\in (0,s]$. By the flow property stated in Theorem 3.3 \cite{Kun84}, $u_{t,\psi}(s-\varepsilon,x)=\E\left[u_{t,\psi}(s,{\tilde X}^{s-\varepsilon,x}_s)\right]$. On the other hand, by Taylor expansion,
\begin{align*}
   u_{t,\psi}&(s,{\tilde X}^{s-\varepsilon,x}_s)-u_{t,\psi}(s,x)=\nabla_x u_{t,\psi}(s,x).\left(\int_{s-\varepsilon}^s{\tilde \sigma}(r,\E[\alpha(X_r)],{\tilde X}^{s-\varepsilon,x}_r)d{\tilde W}_r+\int_{s-\varepsilon}^sb(r,\E[\alpha(X_r)],{\tilde X}^{s-\varepsilon,x}_r)dr\right)\\&+\frac{1}{2}\sum_{j,k=1}^n\partial_{jk}u_{t,\psi}(s,x)\left(\sum_{l=0}^{\tilde d}\int_{s-\varepsilon}^s{\tilde \sigma}^l_j(r,\E[\alpha(X_r)],{\tilde X}^{s-\varepsilon,x}_r)d{\tilde W}^l_r\right)\left(\sum_{m=0}^{\tilde d}\int_{s-\varepsilon}^s{\tilde \sigma}^m_k(r,\E[\alpha(X_r)],{\tilde X}^{s-\varepsilon,x}_r)d{\tilde W}^m_r\right)+R_{\varepsilon}
\end{align*}
where $R_\varepsilon$ is some random reminder such that $|R_\varepsilon|\le C|{\tilde X}^{s-\varepsilon,x}_s-x|^3$ with $C$ a deterministic finite constant. Since $\E\left[\int_{s-\varepsilon}^s{\tilde \sigma}(r,\E[\alpha(X_r)],{\tilde X}^{s-\varepsilon,x}_r)d{\tilde W}_r\right]=0$ and, by It\^o's isometry, \begin{align*}
   \E\bigg[\left(\sum_{l=1}^{\tilde d}\int_{s-\varepsilon}^s{\tilde \sigma}^l_j(r,\E[\alpha(X_r)],{\tilde X}^{s-\varepsilon,x}_r)d{\tilde W}^l_r\right)&\left(\sum_{m=1}^{\tilde d}\int_{s-\varepsilon}^s{\tilde \sigma}^m_k(r,\E[\alpha(X_r)],{\tilde X}^{s-\varepsilon,x}_r)d{\tilde W}^m_r\right)\bigg]\\&=\E\bigg[\int_{s-\varepsilon}^sa_{jk}(r,\E[\alpha(X_r)],{\tilde X}^{s-\varepsilon,x}_r)dr\bigg],
\end{align*} we deduce that
\begin{align*}
   &\bigg|\frac{u_{t,\psi}(s,x)- u_{t,\psi}(s-\varepsilon,x)}{\varepsilon}+\nabla_x u_{t,\psi}(s,x).b(s,\E[\alpha(X_s)],x)+\frac{1}{2}\sum_{j,k=1}^n\partial_{jk}u_{t,\psi}(s,x)a_{jk}(s,\E[\alpha(X_s)],x)\bigg|\\
&\le C\bigg(\left|b(s,\E[\alpha(X_s)],x)-\frac{1}{\varepsilon}\E\left[\int_{s-\varepsilon}^sb(r,\E[\alpha(X_r)],{\tilde X}^{s-\varepsilon,x}_r)dr\right]\right|+\left|a(s,\E[\alpha(X_s)],x)-\frac{1}{\varepsilon}\E\left[\int_{s-\varepsilon}^sa(r,\E[\alpha(X_r)],{\tilde X}^{s-\varepsilon,x}_r)dr\right]\right|\\&+\frac{\E[|{\tilde X}^{s-\varepsilon,x}_s-x|^3]}{\varepsilon}+\frac{1}{\varepsilon}\sum_{j,k=1}^n\E\bigg[\bigg|\int_{s-\varepsilon}^sb_j(r,\E[\alpha(X_r)],{\tilde X}^{s-\varepsilon,x}_r)dr\sum_{m=0}^{\tilde d}\int_{s-\varepsilon}^s{\tilde \sigma}^m_k(r,\E[\alpha(X_r)],{\tilde X}^{s-\varepsilon,x}_r)d{\tilde W}^m_r \bigg|\bigg]\bigg).
\end{align*}
By continuity and uniform integrability of $r\mapsto [{\tilde \sigma},b](r,\E[\alpha(X_r)],{\tilde X}^{s-\varepsilon,x}_r)$ on $[s-\varepsilon,T]$, the two first terms in the right-hand side converge to $0$ as $\varepsilon\to 0$. Moreover the expectations in the two last terms are smaller than $C\varepsilon^{3/2}$. Taking into account the continuity with respect to $s$ of $\nabla_x u_{t,\psi}(s,x).b(s,\E[\alpha(X_s)],x)+\frac{1}{2}\sum_{j,k=1}^n\partial_{jk}u_{t,\psi}(s,x)a_{ik}(s,\E[\alpha(X_s)],x)$, we conclude that $u_{t,\psi}$ is a classical solution of the Feynman-Kac PDE \eqref{fk}.
\end{proof}

We are now ready to prove Theorem \ref{thmbias}.
\begin{proof}[Proof of Theorem \ref{thmbias}]
Let $t\in[0,T]$. Applying It\^{o}'s formula and taking into account the Feynman-Kac PDE \eqref{fk}, we obtain that
\begin{align*}
   u_{t,\psi}(t,X^1_t)&-u_{t,\psi}(t,X^{1,N,h}_t)=u_{t,\psi}(0,X^1_0)-u_{t,\psi}(0,X^1_0)+\int_0^t\nabla_x u_{t,\psi}(s,X^1_s).\sigma(s,\E[\alpha(X_s)],X^1_s)dW^1_s\\&-\int_0^t\nabla_x u_{t,\psi}(s,X^{1,N,h}_s).\sigma\left(\tau^h_s,\frac{1}{N}\sum_{i=1}^N\alpha(X^{i,N,h}_{\tau^h_s}),X^{1,N,h}_{\tau^h_s}\right)dW^1_s\\
&+\int_0^t\nabla_x u_{t,\psi}(s,X^{1,N,h}_s).\left(b\left(s,\E[\alpha(X_s)],X^{1,N,h}_s\right)-b\left({\tau^h_s},\frac{1}{N}\sum_{i=1}^N\alpha(X^{i,N,h}_{\tau^h_s}),X^{1,N,h}_{\tau^h_s}\right)\right)ds\\&+\frac{1}{2}\sum_{j,k=1}^n\int_0^t
\partial_{jk}u_{t,\psi}(s,X^{1,N,h}_s)\left(a_{jk}(s,\E[\alpha(X_s)],X^{1,N,h}_s)-a_{jk}\bigg({\tau^h_s},\frac{1}{N}\sum_{i=1}^N\alpha(X^{i,N,h}_{\tau^h_s}),X^{1,N,h}_{\tau^h_s}\bigg)\right)ds.
\end{align*}
Integrability deduced from Propositions \ref{propsznit} and \ref{propregufk} and the properties of $\sigma$ ensure that the expectations of the stochastic integrals vanish. Therefore, setting for $s\in[0,t]$
\begin{align*}e^\psi(s)&:=\E\bigg[\nabla_x u_{t,\psi}(s,X^{1,N,h}_s).\left(b\left(s,\E[\alpha(X_s)],X^{1,N,h}_s\right)-b\left(\tau^h_s,\frac{1}{N}\sum_{i=1}^N\alpha(X^{i,N,h}_{\tau^h_s}),X^{1,N,h}_{\tau^h_s}\right)\right)\bigg]\\
   e_{jk}^\psi(s)&:=\E\bigg[
\partial_{jk}u_{t,\psi}(s,X^{1,N,h}_s)\left(a_{jk}(s,\E[\alpha(X_s)],X^{1,N,h}_s)-a_{jk}\left(\tau^h_s,\frac{1}{N}\sum_{i=1}^N\alpha(X^{i,N,h}_{\tau^h_s}),X^{1,N,h}_{\tau^h_s}\right)\right)\bigg],\;1\le j,k\le n,
\end{align*}
and using that $u_{t,\psi}(t,.)=\psi(.)$, we obtain
\begin{align}
   \E\left[\psi(X_t)\right]&-\E[\psi(X^{1,N,h}_t)]=\int_0^t\bigg(e^\psi(s)+\frac{1}{2}\sum_{j,k=1}^n e_{jk}^\psi(s)\bigg)ds.\label{difesp}
\end{align}
Let us now estimate $e_{jk}^\psi(s)$ :
\begin{align*}
   e_{jk}^\psi(s)&=\E\bigg[\partial_{jk}u_{t,\psi}(s,X^{1,N,h}_s)\left(a_{jk}(s,\E[\alpha(X_s)],X^{1,N,h}_s)-a_{jk}(\tau^h_s,\E[\alpha(X_{\tau^h_s})],X^{1,N,h}_s)\right)\bigg]\\&+\E\bigg[\left(\partial_{jk}u_{t,\psi}(s,X^{1,N,h}_s)-\partial_{jk}u_{t,\psi}(s,X^{1,N,h}_{\tau^h_s})\right)\left(a_{jk}(\tau^h_s,\E[\alpha(X_{\tau^h_s})],X^{1,N,h}_s)-a_{jk}(\tau^h_s,\E[\alpha(X_{\tau^h_s})],X^{1,N,h}_{\tau^h_s})\right)\bigg]\\&+\E\bigg[\partial_{jk}u_{t,\psi}(s,X^{1,N,h}_{\tau^h_s})\left(a_{jk}(\tau^h_s,\E[\alpha(X_{\tau^h_s})],X^{1,N,h}_s)-a_{jk}(\tau^h_s,\E[\alpha(X_{\tau^h_s})],X^{1,N,h}_{\tau^h_s})\right)\bigg]\\&+\E\bigg[\partial_{jk}u_{t,\psi}(s,X^{1,N,h}_{s})\left(a_{jk}(\tau^h_s,\E[\alpha(X_{\tau^h_s})],X^{1,N,h}_{\tau^h_s})-a_{jk}\left(\tau^h_s,\E[\alpha(X^{1,N,h}_{\tau^h_s})],X^{1,N,h}_{\tau^h_s}\right)\right)\bigg]\\&+\E\bigg[\partial_{jk}u_{t,\psi}(s,X^{1,N,h}_{s})\left(a_{jk}(\tau^h_s,\E[\alpha(X^{1,N,h}_{\tau^h_s})],X^{1,N,h}_{\tau^h_s})-a_{jk}\left(\tau^h_s,\frac{1}{N}\sum_{i=1}^N\alpha(X^{i,N,h}_{\tau^h_s}),X^{1,N,h}_{\tau^h_s}\right)\right)\bigg].
\end{align*}
Let us respectively denote by $e_{jk}^{1,\psi}(s)$, $e_{jk}^{2,\psi}(s)$, $e_{jk}^{3,\psi}(s)$, $e_{jk}^{4,\psi}(s)$ and $\bar{e}_{jk}^{\psi}(s)$ the five terms in the right-hand side. Since, by Proposition \ref{propregufk}, $\partial_{jk}u_{t,\psi}$ is bounded by a finite constant $M^{(2)}_\psi$ not depending in $t$ and $a$ is Lipschitz continuous with constant $L_a$, $|e_{jk}^{1,\psi}(s)|\le M^{(2)}_\psi L_a\left((s-\tau^h_s)+|\E[\alpha(X_s)]-\E[\alpha(X_{\tau^h_s})]|\right)$. Since $\alpha$ is ${\mathcal C}^2$ with bounded derivatives and Lipschitz continuous second order derivatives, for $0\le r\le s\le T$, computing $\alpha(X_s)-\alpha(X_{r})$ by It\^o's formula, taking expectations and remarking that the expectation of the stochastic integral vanishes, we obtain the existence of a finite constant $C$ not depending on $r$ and $s$ such that $|\E[\alpha(X_s)]-\E[\alpha(X_{r})]|\le C (s-r)$. Hence
$|e_{jk}^{1,\psi}(s)|\le M^{(2)}_\psi L_a(1+C)(s-\tau^h_s)\le M^{(2)}_\psi L_a(1+C)h$.

Since, by Proposition \ref{propregufk}, $\partial_{jk}u_{t,\psi}$ is Lipschitz continuous in space with constant $L^{(2)}_\psi$ not depending on $t$ and $a_{jk}$ is Lipschitz continuous in space with constant $L_a$ we have $|e_{jk}^{2,\psi}(s)|\le L^{(2)}_\psi L_a\E[|X^{1,N,h}_s-X^{1,N,h}_{\tau^h_s}|^2]$. Combining the fact that $\sup_{N\ge 1,h\ge 0,s\in[0,T]}\E[|X^{1,N,h}_s|^2]<\infty$ deduced from from Proposition \ref{propsznit}, the Lipschitz continuity of $\alpha$ and the affine growth of $\sigma,b$ in their $p+n$ last variables uniform in their first variable, one easily checks that there is a finite constant $C$ not depending on $N$ and $h$ such that for all $0\le r\le s\le T$, $\E[|X^{1,N,h}_s-X^{1,N,h}_{r}|^2]\le C(s-r)$ and deduce that $|e_{jk}^{2,\psi}(s)|\le L^{(2)}_\psi L_a Ch$.

Remarking that the first and second order derivatives of $a_{jk}$ with respect to its $n$ last variables have affine growth under our assumptions, computing $a_{jk}(\tau^h_s,\E[\alpha(X_{\tau^h_s})],X^{1,N,h}_s)-a_{jk}(\tau^h_s,\E[\alpha(X_{\tau^h_s})],X^{1,N,h}_{\tau^h_s})$ by It\^o's formula, multiplying by the random variable $\partial_{jk}u_{t,\psi}(s,X^{1,N,h}_{\tau^h_s})$ bounded by $M^{(2)}_\psi$, taking expectations and remarking that the contribution of the stochastic integral vanishes, we obtain that $|e_{jk}^{3,\psi}(s)|\le Ch$ with $C$ not depending on $N,s,t,h$. Last, $|e_{jk}^{4,\psi}(s)|\le M^{(2)}_\psi L_a\left|\E[\alpha(X_{\tau^h_s})]-\E[\alpha(X^{1,N,h}_{\tau^h_s})]\right|$. Hence there is a finite constant $C$ not depending on $N,s,t,h$ such that
\begin{equation}\label{part1ejkpsi}
 \forall 0\le s\le t\le T,\;|e_{jk}^\psi(s)|\le C\left(h+\left|\E[\alpha(X_{\tau^h_s})]-\E[\alpha(X^{1,N,h}_{\tau^h_s})]\right|\right)+|\bar{e}_{jk}^{\psi}(s)|.\end{equation} Let us now estimate $|\bar{e}_{jk}^{\psi}(s)|$. Denoting by $\nabla_2a_{jk}$ the partial gradient of $a$ with respect to its variables with indices in $\{2,\hdots,1+p\}$, we have
\begin{align*}
\bar{e}_{jk}^{\psi}(s)=&\E\bigg[\partial_{jk}u_{t,\psi}(s,X^{1,N,h}_s)\bigg(\E[\alpha(X^{1,N,h}_{\tau^h_s})]-\frac{1}{N}\sum_{i=1}^N\alpha(X^{i,N,h}_{\tau^h_s})\bigg)\\
   &\phantom{\E\bigg[}.\int_0^1 \left\{\nabla_2a_{jk}\bigg({\tau^h_s},v\E[\alpha(X_{\tau^h_s}^{1,N,h})]+\frac{1-v}{N}\sum_{i=1}^N\alpha(X^{i,N,h}_{\tau^h_s}),X^{1,N,h}_{\tau^h_s}\bigg)-\nabla_2a_{jk}({\tau^h_s},\E[\alpha(X^{1,N,h}_{\tau^h_s})],X^{1,N,h}_{\tau^h_s})\right\}dv\bigg]\\&+\E\bigg[\bigg(\E[\alpha(X^{1,N,h}_{\tau^h_s})]-\frac{1}{N}\sum_{i=1}^N\alpha(X^{i,N,h}_{\tau^h_s})\bigg).\partial_{jk}u_{t,\psi}(s,X^{1,N,h}_{s})\nabla_2a_{jk}({\tau^h_s},\E[\alpha(X^{1,N,h}_{\tau^h_s})],X^{1,N,h}_{\tau^h_s})\bigg].
\end{align*}
Let us respectively denote by $e_{jk}^{5,\psi}(s)$ and $e_{jk}^{6,\psi}(s)$ the two terms in the right-hand side.
By Proposition \ref{propregufk}, $\partial_{jk}u_{t,\psi}$ is bounded by $M^{(2)}_\psi$ and Lipschitz continuous in space with constant $L^{(2)}_\psi$ with $(M^{(2)}_\psi,L^{(2)}_\psi)$ not depending on $t$. By assumption $\nabla_2 a_{jk}$ is bounded by $M^{(2)}_a$ and  Lipschitz continuous in its $p+n$ last variables with constant $L^{(2)}_{a}$. Therefore, 
\begin{align}
   \frac{1}{M^{(2)}_\psi L^{(2)}_{a}}&|e_{jk}^{5,\psi}(s)|\le \frac{1}{2}\E\bigg[\bigg|\frac{1}{N}\sum_{i=1}^N\alpha(X^{i,N,h}_{\tau^h_s})-\E[\alpha(X_{\tau^h_s}^{1,N,h})]\bigg|^2\bigg]\notag\\&\le \E\bigg[\bigg|\frac{1}{N}\sum_{i=1}^N(\alpha(X^{i,N,h}_{\tau^h_s})-\alpha(X^{i,h}_{\tau^h_s}))-\E[\alpha(X^{1,N,h}_{\tau^h_s})-\alpha(X^{1,h}_{\tau^h_s})]\bigg|^2\bigg]+\E\bigg[\bigg|\frac{1}{N}\sum_{i=1}^N\alpha(X^{i,h}_{\tau^h_s})-\E[\alpha(X^{1,h}_{\tau^h_s})]\bigg|^2\bigg].\notag
  \end{align}
 
The second term in the right-hand side of the inequality is the variance of the empirical mean of i.i.d. random variables. It is therefore equal to $ \frac{1}{N}\left(\E[|\alpha(X^{1,h}_{\tau^h_s})|^2]-|\E[\alpha(X^{1,h}_{\tau^h_s})]|^2\right)$. The first term is also a variance and we upper-bound it by the corresponding second order moment, which is not greater than $\frac{1}{N}\sum_{i=1}^N\E\left[\left|\alpha(X^{i,h}_{\tau^h_s})-\alpha(X^{i,N,h}_{\tau^h_s})\right|^2\right]$ according to Jensen's inequality for the empirical mean. Therefore, denoting by $L_\alpha$ the Lipschitz constant of $\alpha$, we have 

 \begin{align}
 |e_{jk}^{5,\psi}(s)| &\le M^{(2)}_\psi L^{(2)}_{a} \left(\frac{1}{N}\sum_{i=1}^N\E\left[\left|\alpha(X^{i,h}_{\tau^h_s})-\alpha(X^{i,N,h}_{\tau^h_s})\right|^2\right]+\frac{1}{N}\left(\E[|\alpha(X^{1,h}_{\tau^h_s})|^2]-|\E[\alpha(X^{1,h}_{\tau^h_s})]|^2\right) \right) \le L_\alpha^2\; M^{(2)}_\psi L^{(2)}_{a} \frac{C}{N},\label{estimoyempesp}
\end{align}
with $C$ finite and not depending on $N,s,t$ according to Proposition \ref{propsznit}.
Since $\E[\alpha(X^{1,N,h}_{\tau^h_s})]-\frac{1}{N}\sum_{i=1}^N\alpha(X^{i,N,h}_{\tau^h_s})$ is centered and by exchangeability of $(X^{i,N,h})_{1\le i\le N}$, we may replace  $\partial_{jk}u_{t,\psi}(s,X^{1,N,h}_{s})\nabla_2a_{jk}(\tau^h_s,\E[\alpha(X^{1,N,h}_{\tau^h_s})],X^{1,N,h}_{\tau^h_s})$ by $\frac{1}{N}\sum_{i=1}^N\partial_{jk}u_{t,\psi}(s,X^{i,N,h}_{s})\nabla_2a_{jk}(\tau^h_s,\E[\alpha(X^{1,N,h}_{\tau^h_s})],X^{i,N,h}_{\tau^h_s})-\E[\partial_{jk}u_{t,\psi}(s,X^{1,N,h}_{s})\nabla_2a_{jk}(\tau^h_s,\E[\alpha(X^{1,N,h}_{\tau^h_s})],X^{1,N,h}_{\tau^h_s})]$ in the expectation defining $e_{jk}^{6,\psi}(s)$. With the Cauchy-Schwarz inequality, we deduce that
\begin{align*}
   &|e_{jk}^{6,\psi}(s)|\le\left(\E\bigg[\bigg|\frac{1}{N}\sum_{i=1}^N\alpha(X^{i,N,h}_{\tau^h_s})-\E[\alpha(X_{\tau^h_s})]\bigg|^2\bigg]\right)^{1/2}\Bigg(\E\bigg[\bigg|\frac{1}{N}\sum_{i=1}^N\partial_{jk}u_{t,\psi}(s,X^{i,N,h}_{s})\nabla_2a_{jk}(\tau^h_s,\E[\alpha(X^{1,N,h}_{\tau^h_s})],X^{i,N,h}_{\tau^h_s})\\&\phantom{|e_{jk}^{6,\psi}(s)|\le\E^{1/2}\bigg[\bigg|\frac{1}{N}\sum_{i=1}^N\alpha(X^{i,N,h}_{\tau^h_s})-\E[\alpha(X_{\tau^h_s})]\bigg|^2\bigg]\E^{1/2}}-\E\left[\partial_{jk}u_{t,\psi}(s,X^{1,N,h}_{s})\nabla_2a_{jk}(\tau^h_s,\E[\alpha(X^{1,N,h}_{\tau^h_s})],X^{1,N,h}_{\tau^h_s})\right]\bigg|^2\bigg]\Bigg)^{1/2}.
\end{align*}
Since $\R^n\ni x\mapsto\partial_{jk}u_{t,\psi}(s,x)$ (resp. $\R^n\ni x\mapsto\nabla_2a_{jk}(\tau^h_s,\E[\alpha(X^{1,N,h}_{\tau^h_s})],x)$) is bounded by $M^{(2)}_\psi$ (resp. $M^{(2)}_a$) and 
Lipschitz continuous with constant $L^{(2)}_\psi$ (resp. $L^{(2)}_a$), reasoning like in the derivation of \eqref{estimoyempesp}, we obtain that the second factor in the right-hand side is smaller than $(M^{(2)}_\psi L^{(2)}_a+M^{(2)}_a L^{(2)}_\psi)\frac{C}{\sqrt{N}}$ so that $|e_{jk}^{6,\psi}(s)|\le L_\alpha(M^{(2)}_\psi L^{(2)}_a+M^{(2)}_a L^{(2)}_\psi)\frac{C}{N}$
with $C$ finite and not depending on $N,s,t,\psi$. With \eqref{part1ejkpsi}, we deduce that
$$\forall 0\le s\le t\le T,\;|e_{jk}^\psi(s)|\le C\left(\frac{1}{N}+h+\left|\E[\alpha(X_{\tau^h_s})]-\E[\alpha(X^{1,N,h}_{\tau^h_s})]\right|\right).$$
Estimating $e^\psi(s)$ in a similar way and using \eqref{difesp}, we deduce the existence of a finite constant $C$ not depending on $N$ and $h$ such that
\begin{equation}
  \forall t\in[0,T],\;\sup_{s\in[0,t]}\left|\E\left[\psi(X_s)\right]-\E[\psi(X^{1,N,h}_s)]\right|\le C\left(\frac{1}{N}+h\right)+C\int_0^t\left|\E[\alpha(X_{\tau^h_s})]-\E[\alpha(X^{1,N,h}_{\tau^h_s})]\right|ds.\label{difpsialp}
\end{equation}
Summing for $j\in\{1,\hdots,p\}$ this inequality applied with $\psi$ equal to the $j$-th coordinate of $\alpha$, remarking that Proposition \ref{propsznit} and the Lipschitz continuity of $\alpha$ ensure that $\sup_{t\in[0,T]}\left|\E[\alpha(X^{i,N,h}_t)]-\E[\alpha(X_t)]\right|<\infty$ and applying Gronwall's lemma, we deduce that $\sup_{t\in[0,T]}\left|\E[\alpha(X_t)]-\E[\alpha(X^{1,N,h}_t)]\right|\le C\left(\frac{1}{N}+h\right)$. We conclude by combining this estimation and \eqref{difpsialp}.
\end{proof}

\begin{remark}
\begin{itemize}
   \item A careful look at the proof shows that the Lipschitz continuity of $a$ and $b$ in the time variable is not needed to obtain that $\sup_{t\in[0,T]}|\E[\psi(X^{1,N}_t)]-\E[\psi(X_t)]|\le \frac{C}{N}$. Indeed, this property is only used to estimate $e^{1,\psi}_{jk}(s)$ which vanishes when $h=0$ since $\tau^0_s=s$. Moreover, if $a$ and $b$ are only H\"older continuous with exponent $\alpha\in (0,1)$ in the time variable, then the above estimation deteriorates to $\sup_{t\in[0,T]}|\E[\psi(X^{1,N,h}_t)]-\E[\psi(X_t)]|\le C\left(\frac{1}{N}+h^\alpha\right)$.
 \item On the other hand, when there is no nonlinearity in the sense of McKean in the SDE \eqref{edsnl} ($p=0$), we obtain that the weak order of convergence of the Euler scheme is $1$ under mere global Lipschitz continuity of $a,b$, continuity of $\sigma,b$ and ${\mathcal C^2}$ regularity in space  of $\sigma,b,\psi$ with bounded and Lipschitz derivatives. The key step is that the decomposition of $e^{2,\psi}_{jk}(s)+e^{3,\psi}_{jk}(s)$ avoids to differentiate the solution to the Feynman-Kac partial differential equation more than two times in space and only requires Lipschitz continuity of the second order spatial derivatives.
\end{itemize}
\end{remark}

\section{\label{sec:num}Numerical Experiments }

 We conduct two types of numerical tests. First, we estimate the bias using regular Monte-Carlo for examples of one dimensional ($n=1$) mean-field SDEs taken from \cite{KT16} to provide numerical evidence of the $\mathcal{O}(N^{-1})$ behaviour of the bias for a fixed value of the time step h. Then we present the antithetic sampling results on these same examples.
 
The code for running these experiments has a number of iterations as an input parameter. This latter is useful to observe the behaviour of the bias when increasing the number of particles. Therefore, we give an initial number of particles  that we multiply by two from an iteration to the other. Except for the polynomial drift and the plane rotator examples where it is respectively equal to $8$ and $4$, the number of iterations chosen is equal to five and the initial number of particles is twenty so that the final number of particles is $320$. The simulation is done with $5.10^6$ runs except for the Plane rotator example where we push further the number of Monte Carlo runs up to $4.9 \times 10^8$. 

We also define the precision as half the width of the $95\%$ confidence interval of the empirical mean  i.e. \textit{Precision $= 1.96 \times \sqrt[]{\frac{\text{Variance}}{\text{number of runs}}}$} where \textit{Variance} denotes the empirical variance over the runs of the empirical mean over the particles.

In order to measure the relevance of the antithetic sampling technique for variance reduction, we compute the variance of the difference between the empirical mean of the system with $2N$ particles and the mean of the empirical means of the two independent systems with N particles:
\begin{align}\label{AntiVar}
    \text{Var}\left[ \frac{1}{2N} \sum \limits_{i=1}^{2N}  \psi(X^{i,2N}_T) - \frac{1}{2N} \sum \limits_{i=1}^N  \left( \psi(X^{i,N}_T) + \psi(Y^{i,N}_T) \right) \right].
\end{align}
Here $(Y^{i,N}_t)_{1 \leq i \leq N}$ is the system with $N$ particles driven by $(X^{i}_0,W^{i})_{N+1 \leq i \leq 2N}$.

\subsection{Generalised Ornstein-Uhlenbeck process}
\subsubsection*{Model}
We consider the following generalization of the Ornstein-Uhlenbeck SDE to a linear mean-field SDE:
\begin{align*}
	dX_t = \Big[ \gamma X_t + \beta \mathbb{E}[X_t]\Big]dt + \upsilon dW_t \text{,  with  } X(0) = x_0
\end{align*}
for parameters $\gamma$, $\beta$, $\upsilon$ $\in$  $\mathbb{R}$ and initial data $x_0$ $\in \mathbb{R}$. 

The functions $\alpha(x) = x$, $b(t,y, x) = \gamma x + \beta y$ and $\sigma(t,y, x) = \upsilon$ therefore satisfy the hypotheses of regularity of Theorem \ref{thmbias}.

The first and second moments of $X_t$ are respectively given by $\mathbb{E}[X_t] = x_0 \exp\Big( (\gamma + \beta)t\Big)$ and $\mathbb{E}[X_t^2] =  x_{0}^2 \exp\Big( 2(\gamma + \beta) t\Big) + \frac{\upsilon^2}{2\gamma}\Big(\exp(2\gamma t) - 1 \Big)$.

The associated particle approximation of the SDE is given by the following system:
\begin{align*}
	d{X}^{i,N}_t =  \Big[ \gamma {X}^{i,N}_t + \frac{\beta}{N} \sum \limits_{j=1}^N {X}^{j,N}_t \Big]dt + \upsilon dW^{i}_t  \; , \; 1 \leq i \leq  N \; \text{with } X^{i,N}_0 = x_0
\end{align*}
where $\Big((W^{i}_t)_{t \geq 0}\Big)_{1 \leq i \leq N}$ are independent Brownian motions.

Because of the linearity of the drift coefficient, it is possible to obtain closed form expressions for $\mathbb{E}[X^{1,N,h}_t]$ and $\mathbb{E}[(X^{1,N,h}_t)^2]$ and deduce the asymptotic behaviour of the biases of the first and second order moments as $N \to \infty$ and $h \to 0$. 

Let $h=T/K$ for some $K\in\N^*$. For $k\in\{0,\hdots,K-1\}$ and $i\in\{1,\hdots,N\}$, we have
\begin{align*}
   X^{i,N,h}_{(k+1)h}&=(1+\gamma h)X^{i,N,h}_{kh}+\beta h\bar{X}^{N,h}_{kh}+v(W^i_{(k+1)h}-W^i_{kh})\mbox{ with }\bar{X}^{N,h}_{kh}=\frac{1}{N}\sum_{j=1}^NX^{j,N,h}_{kh}.\end{align*}
Hence $\E[X^{i,N,h}_{kh}]=(1+(\gamma+\beta) h)^kx_0$ so that 
$$\E[X_T]-\E[X^{i,N,h}_{T}]=\frac{1}{2}
{(\gamma+\beta)^2Te^{(\gamma+\beta)T}}h x_0 +\mathcal{O}(h^2) \;\; \text{as $h\to 0$.} $$ 

To compute $\E[(X^{1,N,h}_{kh})^2]$, we remark that for $k\in\{0,\hdots,K-1\}$\begin{align*}
\E[(X^{1,N,h}_{(k+1)h})^2]&=(1+\gamma h)^2\E[(X^{1,N,h}_{kh})^2]+(2(1+\gamma h)+\beta h)\beta h\left(\frac{1}{N}\E[(X^{1,N,h}_{kh})^2]+\frac{N-1}{N}\E[X^{1,N,h}_{kh}X^{2,N,h}_{kh}]\right)+v^2h\\
\E[X^{1,N,h}_{(k+1)h}X^{2,N,h}_{(k+1)h}]&=(1+\gamma h)^2\E[X^{1,N,h}_{kh}X^{2,N,h}_{kh}]+(2(1+\gamma h)+\beta h)\beta h\left(\frac{1}{N}\E[(X^{1,N,h}_{kh})^2]+\frac{N-1}{N}\E[X^{1,N,h}_{kh}X^{2,N,h}_{kh}]\right)\end{align*}
Since $X^{i,N,h}_0=x_0$ for all $i\in\{1,\hdots,N\}$, subtracting the two last equations, we obtain
\begin{align*}
   \forall k\in\{1,\hdots, K\},\;\E[(X^{1,N,h}_{kh})^2]-\E[X^{1,N,h}_{kh}X^{2,N,h}_{kh}]=\frac{(1+\gamma h)^{2k}-1}{2\gamma +\gamma^2h}\times v^2
\end{align*} and deduce that
$$\E[(X^{1,N,h}_{(k+1)h})^2]=(1+(\gamma+\beta)h)^2\E[(X^{1,N,h}_{kh})^2]+(2(1+\gamma h)+\beta h)\beta \frac{1-N}{N}\times \frac{(1+\gamma h)^{2k}-1}{2\gamma +\gamma^2h}\times v^2h+v^2h.$$
We conclude that
\begin{align*}
   \E[(X^{1,N,h}_{kh})^2]=&(1+(\gamma+\beta)h)^{2k}x_{0}^2+\frac{(1+(\gamma+\beta) h)^{2k}-1}{2(\gamma+\beta)+(\gamma+\beta)^2h}\left(1+\frac{N-1}{N}\times\frac{2\beta+(2\beta \gamma +\beta^2)h}{2\gamma +\gamma^2h}\right)v^2\\&+\frac{1-N}{N}\times\frac{(1+(\gamma+\beta) h)^{2k}-(1+\gamma h)^{2k}}{2\gamma +\gamma^2h}\times v^2\\
=&(1+(\gamma+\beta)h)^{2k}x_{0}^2+\frac{N-1}{N}\times\frac{(1+\gamma h)^{2k}-1}{2\gamma +\gamma^2h}\times v^2+\frac{1}{N}\times\frac{(1+(\gamma+\beta) h)^{2k}-1}{2(\gamma+\beta)+(\gamma+\beta)^2h}\times v^2
\end{align*}
so that $\E[(X_T)^2]-\E[(X^{1,N,h}_{T})^2]=\left((\gamma+\beta)^2Te^{2(\gamma+\beta)T}x_{0}^2+\frac{e^{2\gamma T}-1+2\gamma T e^{2\gamma T}}{4}v^2\right)h+\left(\frac{e^{2\gamma T}-1}{2\gamma}+\frac{1-e^{2(\gamma+\beta)T}}{2(\gamma+\beta)}\right)\frac{v^2}{N}+{\mathcal O}(h^2+\frac{h}{N}+\frac{1}{N^2})$ as $h\to 0$ and $N\to\infty$. Moreover,
$$\E[(X^h_T)^2]-\E[(X^{1,N,h}_{T})^2]=\frac{1}{N}\left(\frac{(1+\gamma h)^{2k}-1}{2\gamma +\gamma^2h}+\frac{1-(1+(\gamma+\beta) h)^{2k}}{2(\gamma+\beta)+(\gamma+\beta)^2h}\right)\times v^2.$$
 
The bias of the time discretized second order moment is exactly of order $1$ in  $\frac{1}{N}$.

\subsubsection*{Results}

In order to illustrate the behaviour of the first and second order moments, we compute the difference between the closed-form discretized moments and the estimated moments and expect the difference to be null up to the statistical error. The results are shown in Tables \ref{tab_OU_1} and \ref{tab_OU_SM} below where we denote by "Difference" that entity. As a test case, we use this model with $\gamma =\frac{1}{2}$, $\beta =\frac{4}{5}$, $\upsilon^2 = \frac{1}{2}$ and $x_{0}=1$. 

\newpage

Concerning the first order moment, we proved above that the bias does not depend on the number $N$ of particles, which is what is observe in the first row of the table below. 
\begin{table}[!h]
	\begin{center}
			\begin{tabular}{|c|c|c|c|c|c|} 
 			\hline
            Nb. particles & 20 & 40 & 80 & 160 & 320 \\ 
 			\hline
 			Estimated first moment & 1.34862 &	1.34861 & 1.34869 &	1.34866 & 1.34866 \\
            \hline
  			Difference & 0.00003 & 0.00004 & -0.00004 & -0.00001	& -0.00001 \\
 			\hline
  			Precision & 0.00016 & 0.00011 & 0.00008 & 0.00006 &	0.00004 \\
 			\hline
		\end{tabular}
        \caption{Generalised Ornstein-Uhlenbeck SDE: Comparison of the estimated first moments with the closed-form discretized value $1.34865$ as well as the associated precision when increasing the number of particles for a number of $5.10^{6}$ runs, $50$ time steps and a time horizon T=1.}
		\label{tab_OU_1}
	\end{center}
\end{table}

As for the second order moment, from the third row we observe that the estimation fits the closed-form discretized value which confirms the behaviour of the bias of order $1$ in $\frac{1}{N}$. 
\begin{table}[!h]
	\begin{center}
		\begin{tabular}{|c|c|c|c|c|c|} 
 			\hline
            Nb. particles & 20 & 40 & 80 & 160 & 320 \\ 
 			\hline
            Closed-form discretized second moment &  2.15552 &	2.14648 &	2.14195	& 2.13969 &	2.13856  \\
 			\hline
 			Estimated second moment & 2.15531 &	2.14655	 & 2.14205 &	2.13970	& 2.13859 \\
            \hline
  			Difference & 0.00021	& -0.00007	& -0.00010 &	-0.00001	& -0.00003 \\
 			\hline
  			Precision & 0.00045 & 0.00032 &	0.00022 & 0.00016 &	0.00011 \\
 			\hline
		\end{tabular}
        \caption{Generalised Ornstein-Uhlenbeck SDE: Comparison of the estimated second moments with their closed-form discretized values as well as the associated precision when increasing the number of particles for a number of $5.10^{6}$ runs, $50$ time steps and a time horizon T=1.}
		\label{tab_OU_SM}
	\end{center}
\end{table}

 Concerning the antithetic sampling of the first order moment, using the linearity of the model, one exactly checks that for $k\in\{0,\hdots,K-1\}$, $\bar X^{2N, h}_{kh} = \frac{1}{2} \left( \bar X^{N, h}_{kh} + \bar Y^{N, h}_{kh} \right)$ where $\bar Y^{N, h}_{kh}$ denotes the empirical mean of the discretized system driven by $(W^{i})_{N+1 \leq i \leq 2N}$. The first row of Table \ref{tab_OU_2} confirms that the variance \eqref{AntiVar} for $\psi(x) = x$ is null.
 
 For the second order moment, we compute the variance \eqref{AntiVar} for $\psi(x) = x^2$ and observe the ratio of decrease $\frac{\text{Variance}(N/2)}{\text{Variance}(N)}$ when increasing the number of particles. The results are shown in the third and forth rows of Table \ref{tab_OU_2}.

\begin{table}[!h]
	\begin{center}
		\begin{tabular}{|c|c|c|c|c|c|} 
 			\hline
            Nb. particles & 20 & 40 & 80 & 160 & 320 \\ 
            \hline
            Variance for the first moment & 7.82256e-32 & 9.52151e-32 & 1.52652e-31 & 2.79857e-31 & 5.39704e-31\\
 			\hline
  			Precision &9.88245e-35 & 1.18553e-34 & 1.89269e-34 & 3.47065e-34 & 6.69575e-34 \\
 			\hline
            Variance for the second moment & 0.000766641 & 0.000191654 & 4.77727e-05 &1.19847e-05 & 2.99537e-06\\
 			\hline
 			Ratio of decrease & $\times$  & 4.00013 &4.01179 &3.98613 & 4.00109\\
 			\hline
  			Precision & 2.4997e-06 & 6.27208e-07 & 1.55659e-07 &3.9365e-08  &9.75389e-09 \\
 			\hline
		\end{tabular}
        \caption{Generalised Ornstein-Uhlenbeck SDE: Evolution of the antithetic variance for both $\psi(x) = x$ and $\psi(x) = x^2$  with its associated precision when increasing the number of particles for a number of $5.10^6$ runs, $50$ time steps and a time horizon of 1.}
		\label{tab_OU_2}
	\end{center}
\end{table}

For the second order moment, the ratio of successive variances is around $4$ which means that the variance of the antithetic estimator is roughly proportional to $N^{-2}$. The antithetic sampling technique therefore shows an important improvement for this diffusion.   

\subsection{Plane Rotator}

\subsubsection*{Model}

The following SDE refers to a model for coupled oscillators in the presence of noise, also known as the Kuramoto model:
\begin{align*}
 	d{X}_t &= \Big[ K \int_{\mathbb{R}} \sin(y - {X}_t)dP_t^{\mu}(y) - \sin({X}_t) \Big]dt + \sqrt[]{2k_B T} dW_t \\
 &= \Big[ K\cos({X}_t)\int_{\mathbb{R}} \sin(y)dP_t^{\mu} - K\sin({X}_t)\int_{\mathbb{R}} \cos(y)dP_t^{\mu} - \sin({X}_t) \Big]dt + \sqrt[]{2k_B T} dW_t
 \end{align*}
 where  ${X}_0$ is distributed according to $\mu$,  $P_t^{\mu}$ denotes the distribution of ${X}_t$, $K > 0$ a coupling parameter, $k_B$ the Boltzmann constant and $T$ the temperature.

The functions $\sigma(x) = \sqrt[]{2k_BT}$ and $\alpha(x) = \begin{bmatrix} \sin(x) \\ \cos(x) \end{bmatrix}$ satisfy the hypotheses of Theorem \ref{thmbias}. One may also find a function $b(t,y,x)$ coinciding with $K(\cos(x)y_1 - \sin(x)y_2) - \sin(x)$ on $ [0,T] \times [-1,1]^2 \times \mathbb{R} $ which satisfies the hypotheses even if $(x,y) \longrightarrow \cos(x)y_1 - \sin(x)y_2$ is not Lipschitz continuous.\\

 The particle system has the following dynamics:
 \begin{align*}
 	d{X}^{i,N}_t = \Big[ K \Big( \cos({X}^{i,N}_t)\frac{1}{N}\sum \limits_{j=1}^N \sin({X}^{j,N}_t) - \sin({X}^{i,N}_t)\frac{1}{N}\sum \limits_{j=1}^N \cos({X}^{j,N}_t) - \sin({X}^{i,N}_t) \Big]dt + \sqrt[]{2k_B T} dW^{i}_t 
\end{align*}
where $\Big((W^{i}_t)_{t \geq 0}\Big)_{1 \leq i \leq N}$ are independent Brownian motions.
 
 \subsubsection*{Results}

 We use this model with $K=1$, $k_{B}T = \frac{1}{8}$ and initial distribution $\mu = \mathcal{N}(\frac{\pi}{4}, \frac{3\pi}{4})$ as a test case. The reference value was computed for $2.5 \times 10^8$ Monte Carlo runs, $1000$ particles and the same input parameters as for the general estimation of the bias (time horizon $T=1$ and time step $h=50$). The value obtained is $0.737576$. The results are shown in Table \ref{tab-PR-1}.

 \begin{table}[!ht]
 	\begin{center}
 		\begin{tabular}{|c|c|c|c|c|c|} 
  			\hline
            Nb. particles & 20 & 40 & 80 & 160 \\ 
  			\hline
  			First moment error & 0.000725 & 0.000355 & 0.000175 & 0.000067 \\
  			\hline
  			Ratio of decrease & $\times$ & 2.04225 & 2.02857 & 2.61194 \\
  			\hline
   			Precision & 4.79156e-05 & 3.38946e-05 & 2.39723e-05 & 1.69533e-05  \\
  			\hline 
 		\end{tabular}
         \caption{Plane Rotator SDE: Evolution of the first order moment errors as well as the associated precision when increasing the number of particles for a number of $4.9 \times 10^{8}$ runs, $50$ time steps and a time horizon T=1.}
		\label{tab-PR-1}
 	\end{center}
 \end{table}

 These results are consistent with a bias proportional to $N^{-1}$. \\

 Table \ref{tab-PR-2} exposes the results obtained for the antithetic variance. 

 \begin{table}[!h]
 	\begin{center}
 		\begin{tabular}{|c|c|c|c|c|c|} 
  			\hline
             Nb. particles & 20 & 40 & 80 & 160  \\ 
  			\hline
  			Variance & 0.000119023 & 3.12526e-05 & 8.01086e-06 & 2.0277e-06 \\
  			\hline
  			Ratio of decrease & $\times$  & 3.80842 & 3.90128 &  3.95071 \\
 			\hline
   			Precision & 2.51877e-08  & 6.75418e-09 & 1.75258e-09 &  4.4622e-10 \\
  			\hline
 		\end{tabular}
         \caption{Plane Rotator SDE: Evolution of the antithetic variance and its associated precision when increasing the number of particles for a number of $4.9 \times 10^8$ runs, $50$ time steps and a time horizon T=1.}
    \label{tab-PR-2}
 	\end{center}
 \end{table}

 The antithetic sampling method is therefore relevant in this example since the ratio of decrease is close to four which means that the variance of the antithetic estimation is roughly proportional to $N^{-2}$. This behaviour has also been observed in Section 3.2.2 \cite{HAT}.

\subsection{Polynomial Drift}

\subsubsection*{Model}

Let us consider the following mean-field SDE:
\begin{align*}
	dX_t = \Big[ \gamma X_t + \mathbb{E}[X_t] - X_t\mathbb{E}[X_t^2] \Big]dt + X_tdW_t \; \text{ with } X(0) = x_0
\end{align*}
for a certain parameter $\gamma \in \mathbb{R} $ and initial data $x_0 \in \mathbb{R}$. 

The function $\sigma(t,y,x) = x$ satisfy the hypotheses of regularity. However, the functions $b(t,y,x) = \gamma x + y_1 - xy_2$ and $\alpha(x) = \begin{bmatrix} x \\ x^2 \end{bmatrix}$ are not Lipschitz.

From the evolution of the Euler discretization $ X^{h}_{(k+1)h} = X^{h}_{kh}\Big( (1 + \gamma h - h\mathbb{E}[(X^{h}_{kh})^2] )\Big) + h \mathbb{E}[X^{h}_{kh}] + X^{h}_{kh}\Big( W_{(k+1)h}-W_{kh}\Big) $, we deduce that:
\begin{align*}
	\mathbb{E}[X^{h}_{(k+1)h}] &= \Big( 1 + (\gamma +1)h - \mathbb{E}[(X^{h}_{kh})^2]h \Big) \mathbb{E}[X^{h}_{kh}] \\
    \mathbb{E}[(X^{h}_{(k+1)h})^2] &= \Big( 1 + \gamma h - \mathbb{E}[(X^{h}_{kh})^2] h \Big)^2 \mathbb{E}[(X^{h}_{kh})^2] + 2\Big( 1 + \gamma h - \mathbb{E}[(X^{h}_{kh})^2] h \Big)h\mathbb{E}^2[X^{h}_{kh}] + \mathbb{E}[(X^{h}_{kh})^2]h + \mathbb{E}^2[X^{h}_{kh}]h^2
\end{align*}
And we solve this system of inductive equations numerically to obtain reference values of the first and second order moments.
 
The idea here, once again, is to approximate $\mathbb{E}[X]$ and $\mathbb{E}[{X}^2]$ by their empirical means in order to define the dynamics of the particle system:
\begin{align*}
	d{X}^{i,N}_t = \Big[ \gamma {X}^{i,N}_t + \frac{1}{N}\sum \limits_{j=1}^N {X}^{j,N}_t - {X}^{i,N}_t \frac{1}{N}\sum \limits_{j=1}^N ({X}^{j,N}_t)^2 \Big]dt + {X}^{i,N}_t dW^{i}_t
\end{align*}
where $\Big((W^{i}_t)_{t \geq 0}\Big)_{1 \leq i \leq N}$ are independent Brownian motions.

\subsubsection*{Results}

We use this model with $\gamma = 2$ and $x= 1$ as a test case. The reference values obtained for the time discretized first and second moments are respectively $1.3845$ and $3.13743$ . For this example, we push the iterations further until 8 so that the final number of particles is 2560. We denote by "Ratio of decrease 1" and "Ratio of decrease 2"  the respective entities $\frac{\text{First moment error}(N/2)}{\text{First moment error}(N)}$ and $\frac{\text{Second moment error}(N/2)}{\text{Second moment error}(N)}$. The results are shown in Tables \ref{tab-PD-1} and \ref{tab-PD-2}.\\

\begin{table}[!h]
	\begin{center}
		\begin{tabular}{|c|c|c|c|c|c|c|c|c|} 
 			\hline
            Nb. particles &  40 & 80 & 160 & 320 & 640 & 1280 & 2560 \\ 
 			\hline
 			First moment error  &	-0.02597&	-0.01561&	-0.00898&	-0.00492&	-0.00261	&-0.00138	&-0.00071 \\
 			\hline
 			Ratio of decrease 1 & $\times$   & 1.66368 & 1.73831 & 1.82520 & 1.88506 & 1.89130 & 1.94366 \\
 			\hline
  			Precision  & 9.348e-05 & 7.402e-05 & 5.779e-05 & 4.431e-05 & 3.330e-05 & 2.459e-05 & 1.789e-05  \\
 			\hline
		\end{tabular}
        \caption{Polynomial Drift SDE: Evolution of the first order moment errors as well as the associated precision when increasing the number of particles for a number of $5.10^{6}$ runs, $50$ time steps and a time horizon T=1}
		\label{tab-PD-1}
	\end{center}
\end{table}

\begin{table}[!ht]
	\begin{center}
		\begin{tabular}{|c|c|c|c|c|c|c|c|c|} 
 			\hline
            Nb. particles &  40 & 80 & 160 & 320 & 640 & 1280 & 2560 \\ 
 			\hline
            Second moment error  & 0.06025 & 0.03575 & 0.02077 & 0.01157 & 0.00637 & 0.00333 & 0.00171 \\
 			\hline
 			Ratio of decrease 2 & $\times$   & 1.68531 & 1.72123 & 1.79516 & 1.81633 & 1.91291 & 1.94737 \\
 			\hline
  			Precision  & 0.0005018 & 0.0003940 & 0.0003029 & 0.0002297 & 0.0001707 & 0.0001251 & 9.063e-05  \\
 			\hline
		\end{tabular}
        \caption{Polynomial Drift SDE: Evolution of the second order moment errors as well as the associated precision when increasing the number of particles for a number of $5.10^{6}$ runs, $50$ time steps and a time horizon T=1}
		\label{tab-PD-2}
	\end{center}
\end{table}

\newpage 

We observe that the ratios of decrease of both the first and the second order moment errors seem to grow as the number of particle increases. It confirms the behaviour of the bias in $\mathcal{O}(\frac{1}{N})$ of Theorem \ref{thmbias} and even tends towards a behaviour of order $1$ in $\frac{1}{N}$ when the number of particle is large.\\

In order to measure the relevance of the antithetic sampling, we compute the variance \eqref{AntiVar} for the first and second order moments.  We denote by "Ratio of decrease V1" and "Ratio of decrease V2"  the respective entities $\frac{\text{Variance for first moment error}(N/2)}{\text{Variance for first moment error}(N)}$ and $\frac{\text{Variance for Second moment error}(N/2)}{\text{Variance for Second moment error}(N)}$. The results are shown in Tables \ref{tab-PD-3} and \ref{tab-PD-4}.

\begin{table}[!h]
	\begin{center}
		\begin{tabular}{|c|c|c|c|c|c|c|c|c|} 
 			\hline
            Nb. particles & 40 & 80 & 160 & 320 & 640 & 1280 & 2560 \\ 
 			\hline
 			Variance for the first moment  & 0.0008910 & 0.0004406 & 0.0001930 & 7.474e-05 & 2.584e-05 & 8.000e-06 & 2.262e-06 \\
 			\hline
 			Ratio of decrease V1 & $\times$   & 2.023 & 2.283 & 2.582 & 2.892 & 3.230  & 3.537\\
 			\hline
  			Precision  & 3.621e-06 & 2.260e-06 & 1.312e-06 & 6.650e-07 & 3.278e-07 & 1.294e-07 & 7.131e-08\\
 			\hline
		\end{tabular}
        \caption{Polynomial Drift SDE: Evolution of the antithetic variance for $\psi(x) = x$ with its associated precision when increasing the number of particles for a number of $5.10^{6}$ runs, $50$ time steps and a time horizon T=1.}
		\label{tab-PD-3}
	\end{center}
\end{table}

\begin{table}[!h]
	\begin{center}
		\begin{tabular}{|c|c|c|c|c|c|c|c|c|} 
 			\hline
            Nb. particles  & 40 & 80 & 160 & 320 & 640 & 1280 & 2560 \\ 
 			\hline
Variance for the second moment  & 0.02380 & 0.01078 & 0.00440 & 0.00170 & 0.00058 & 0.00018 & 5.282e-05  \\
 			\hline
 			Ratio of decrease V2 & $\times$   & 2.2074 & 2.45094 & 2.60197  & 2.89779 & 3.20462 & 3.44723 \\
 			\hline
  			Precision  & 0.000113 & 6.672e-05 & 3.597e-05 &  2.415e-05 & 1.125e-05 & 5.042e-06 & 2.862e-06 \\
 			\hline
		\end{tabular}
        \caption{Polynomial Drift SDE: Evolution of the antithetic variance for $\psi(x) = x^2$  with its associated precision when increasing the number of particles for a number of $5.10^{6}$ runs, $50$ time steps and a time horizon T=1.}
		\label{tab-PD-4}
	\end{center}
\end{table}

The results exposed in Tables \ref{tab-PD-3} and \ref{tab-PD-4} both show that when increasing the number of particles, the ratios of decrease grow gradually. Therefore, the antithetic sampling method may be relevant when simulating with a large number of particles. 

\subsection{Viscous Burgers equation}

\subsubsection*{Model}

Let us consider the following mean-field SDE for a parameter $\upsilon > 0$:
\begin{align*}
	dX_t = \bar{F}_t(X_t)dt + \upsilon dW_t 
\end{align*}
where $\bar{F}_t(x) = \mathbb{P}(X_t \geq x)$.

Using the Fokker-Planck equation satisfied by the density of $X_t$, we show  that $\bar{F}_t(x)$ is solution to the viscous Burgers equation:

		\begin{align*}
			\partial_t V(t,x) = \frac{\upsilon^2}{2} \partial_{xx}V(t,x) - V(t,x)\partial_x V(t,x) 
		\end{align*}

We now suppose that the initial condition is $X(0) = 0$ so that $\bar{F}_0(x) = \mathbf{1}_{\{x \leq 0\}}$. Then the Cole-Hopf transformation gives:
\begin{align*}
	\bar{F}_t(x) = \frac{\mathcal{N}\Big(\frac{t-x}{\upsilon \sqrt[]{t}}\Big)}{\exp\Big( \frac{2x-t}{2\upsilon^2} \Big)\mathcal{N}\Big(\frac{x}{\upsilon \sqrt[]{t}}\Big)+ \mathcal{N}\Big(\frac{t-x}{\upsilon \sqrt[]{t}}\Big)}
\end{align*}
where $\mathcal{N}(x) = \displaystyle \int_{-\infty}^x \exp(- \frac{y^2}{2})\frac{dy}{\sqrt[]{2\pi}}$.

The function $\sigma(t,y,x) = \upsilon$ and $b(t,y,x) = y$ are regular enough to satisfy the hypotheses of Theorem \ref{thmbias}. However, this type of example is not interacting through moments but through a kernel and even a discontinuous one.

We approximate $\bar{F}_t(x)$ by its associated empirical mean $\bar{F}^N_t(x) := \frac{1}{N} \sum \limits_{j=1}^N \mathbf{1}_{\{{X}^{j,N}_t \geq x\}}$ calculated upon $N$ particles which leads to the following dynamics for the particle system:
\begin{align*}
	d{X}^{i,N}_t =  \frac{1}{N} \sum \limits_{j=1}^N \mathbf{1}_{\{{X}^{j,N}_t \geq {X}^{i,N}_t\}}dt + \upsilon dW^{i}_t 
\end{align*}
where $\Big((W^{i}_t)_{t \geq 0}\Big)_{1 \leq i \leq N}$ are independent Brownian motions.

\subsubsection*{Results}

We use this model with $\upsilon = \frac{1}{4}$ and $x_0=0$ as a test case. We estimate the solution $\bar{F}_1(\frac{1}{2})$ of the viscous Burgers equation. Since for $t=1$ and $x=\frac{1}{2}$, $t-x=x$ and $2x -t=0$, one has $\bar{F}_1(\frac{1}{2}) = \frac{1}{2}$.  The results are shown in Table \ref{tab-VB-1}.

\begin{table}[!h]
	\begin{center}
		\begin{tabular}{|c|c|c|c|c|c|} 
 			\hline
            Nb. particles & 20 & 40 & 80 & 160 & 320 \\ 
 			\hline
 			Solution Error  & 0.0141425 & 0.0070329 & 0.00352185 & 0.00174521 & 0.000870359  \\
 			\hline
 			Ratio of decrease & $\times$  & 2.01091 & 1.99693 & 2.01802 & 2.00515 \\
 			\hline
  			Precision &  9.84069e-05 & 6.94588e-05 & 4.90673e-05 & 3.46603e-05 & 2.45057e-05 \\
 			\hline
		\end{tabular}
        \caption{Viscous Burgers equation: Evolution of the solution errors as well as the associated precision when increasing the number of particles for a number of $5.10^{6}$ runs, $500$ time steps and a time horizon T=1.}
		\label{tab-VB-1}
	\end{center}
\end{table}

We observe that the ratio of decrease of the solution error defined by $\frac{\text{Solution error}(N/2)}{\text{Solution error}(N)}$ is consistent with an error proportional to $N^{-1}$ confirming a bias behaviour of order $1$ in $\frac{1}{N}$.\\

As for the antithetic sampling, we compute the corresponding variance \eqref{AntiVar} as well as the related precision and obtain:  

\begin{table}[!h]
	\begin{center}
		\begin{tabular}{|c|c|c|c|c|c|} 
 			\hline
            Nb. particles & 20 & 40 & 80 & 160 & 320 \\ 
 			\hline
 			Variance of the solution & 0.00286895 &	0.00109648 &	0.000413156 	& 0.000155082 &	5.89916e-05  \\
 			\hline
 			Ratio of decrease & $\times$  & 2.61651 & 2.65392 & 2.66411 & 2.62889  \\
 			\hline
  			Precision & 3.98911e-06 & 1.49543e-06 & 5.52988e-07 & 2.04566e-07 & 7.65869e-08  \\
 			\hline
		\end{tabular}
        \caption{Viscous Burgers Equation: Evolution of the antithetic variance and its associated precision when increasing the number of particles for a number of $5.10^{6}$ runs, $500$ time steps and a time horizon T=1.}
		\label{tab-VB-2}
	\end{center}
\end{table}

From Table \ref{tab-VB-2}, we observe that the ratio of decrease of the variance $\frac{\text{Variance of solution}(N/2)}{\text{Variance of solution}(N)}$ is around $2.64$ which is slightly greater than two. Therefore, there is a slight gain in using the antithetic sampling for this type of diffusion.

\newpage 
\pagenumbering{gobble}

\bibliographystyle{plain}

\begin{thebibliography}{6}
\bibitem{BHR} K. Bujok, B. Hambly and C. Reisinger, Multilevel Simulation of Bernoulli Random Variables with Applications to Basket Credit Derivatives, Methodol. Comput. Appl. Probab. 17:579-604, 2015.
\bibitem{CCD}
J.-F. Chassagneux, D. Crisan and F. Delarue, A  Probabilistic approach to classical solutions of the master equation for large population equilibria, Preprint arXiv:1411.3009

\bibitem{CGL} 
N. Chen, P. Glynn and Y. Liu, Estimating expectations of functionals of conditional expectations via multilevel nested simulation, Presentation at MCQMC 2012.
\bibitem{DMPR} P. Del Moral, F. Patras and S. Rubenthaler, Tree based functional expansions for Feynman-Kac particle models, Ann. Appl. Probab. 10(2):778-825, 2009.

\bibitem {Frie75}
A. Friedman. Stochastic Differential Equations and Applications, vol 1. Academic Press,  1975.

\bibitem{GM} C. Graham and S. M\'el\'eard, Stochastic Particle Approximations for Generalized Boltzmann Models and Convergence Estimates, Ann. Probab. 25(1):115-132, 1997.


\bibitem {Gil08}
M. -B. Giles, Multilevel Monte Carlo Path Simulation, Oper. Res. 56(3):607-617, 2008.

\bibitem{Gil15}
M. Giles, Multilevel Monte Carlo methods, Acta Numerica 24:259-328, 2015

\bibitem{Ha}
A.-L. Haji-Ali, Pedestrian flow in the mean-field limit, MSc thesis, KAUST, 2012.


\bibitem{HAT} A.-L. Haji-Ali and R. Tempone, Multilevel and Multi-index Monte Carlo methods for the McKean-Vlasov equation, Stat. Comput., Published online : 12 september 2017.

\bibitem{JMW} B. Jourdain, S. M\'el\'eard and W. Woyczynski, Nonlinear SDEs driven by L\'evy processes and related PDEs.
 ALEA Lat. Am. J. Probab. Math. Stat.  4:1-29, 2008.

\bibitem{Ka}  S. Kanagawa, On the rate of convergence for Maruyama's
  approximate solutions of stochastic differential equations. Yokohama
  Math. J. 36(1):79-86, 1988.

\bibitem{KT} V. Kolokoltsov and M. Troeva, On the mean filed games with common noise and the McKean-Vlasov SPDEs, Preprint arXiv:1506.04594

\bibitem {KT16}
P. Kloeden and T. Shardlow, Gauss-quadrature method for one-dimensional mean-field SDEs. Preprint arXiv:1608.06741, 2016.
\bibitem {Kun84}
H. Kunita, Stochastic differential equations and stochastic flows of diffeormorphisms. {Ecole d'\'{e}t\'{e} de Probabilit\'{e}s de Saint-Flour XII - 1982}. Lecture Notes in Mathematics. Springer Verlag, 1984, pp. 144-305.

\bibitem {Szn91}
A.-S. Snitzman, Topics in propagation of chaos. {Ecole d'\'{e}t\'{e} de Probabilit\'{e}s de Saint-Flour XIX - 1989}. Lecture Notes in Mathematics. Springer Berlin Heidelberg, 1991, pp. 165-251.

\bibitem {tt}D. Talay and L. Tubaro, Expansion of the global error for
numerical schemes solving stochastic differential equations. Stochastic Anal.
Appl. 8(4):483-509, 1990.

\end{thebibliography}

\end{document}